\documentclass[12pt]{amsart}

\usepackage{graphicx}

\newtheorem*{theorem}{Theorem}
\newtheorem*{lemma}{Lemma}
\newtheorem{claim}{Claim}

\begin{document}

\title{Exceptional surgeries on components of two-bridge links}

\author{Kazuhiro Ichihara}
\address{Department of Mathematics, 
College of Humanities and Sciences, 
Nihon University, 3--25--40 Sakurajosui, 
Setagaya, Tokyo 156--8550, Japan}
\email{ichihara@math.chs.nihon-u.ac.jp}

\thanks{The author is partially supported by 
Grant-in-Aid for Scientific Research (No.\ 23740061), 
The Ministry of Education, Culture, Sports, Science and Technology, Japan.}

\date{\today}

\begin{abstract}
In this paper, we give 
a complete classification of exceptional Dehn surgeries on 
a component of a hyperbolic two-bridge link in the 3-sphere.
\end{abstract}

\keywords{Dehn surgery, exceptional surgery, 2-bridge link}

\subjclass[2000]{Primary 57M50; Secondary 57M25}

\maketitle


\section{Introduction}

A Dehn surgery on a link $L$ in a 3-manifold $M$ is defined as 
an operation as; 
take the exterior $E(L)$ of $L$, 
i.e., remove the interior of 
the tubular neighborhood $N(L)$ of $L$ from $M$,  
and then, glue solid tori to $E(L)$. 

One of the motivation to study Dehn surgery is given by 
the fact \cite[Theorem 5.8.2]{Th} due to Thurston: 
On each component of a hyperbolic link, 
there are only finitely many Dehn surgeries 
In view of this, a Dehn surgery on a hyperbolic link 
giving a non-hyperbolic manifold is said to be 
an \textit{exceptional surgery}. 

In the study of exceptional surgery, 
one of the most important problems, 
related to Knot theory, is: 
Completely classify the exceptional surgeries on
hyperbolic links in the 3-sphere $S^3$. 
This seems to be considerably challenging, and 
the problem much easier to tackle is to give 
a complete classification of the exceptional surgeries on 
some class of links. 
Along this line, we consider in this paper 
the hyperbolic \textit{2-bridge links} in $S^3$. 

A link in $S^3$ is called a \textit{2-bridge link} 
if it admits a diagram with exactly two maxima and minima. 
See \cite{Kawauchi} for more details. 
We will follow the definition and notations about 2-bridge link 
from \cite{GodaHayashiSong, YQWu1999}. 
In the following, we denote by {$L_{p/q}$} 
the 2-bridge link associated to a rational number $p/q$.

In this paper, we give 
a complete classification of exceptional surgeries on 
a component of a hyperbolic two-bridge link in $S^3$.

To state our result, we set our notation as follows. 
For a knot $K$ in $S^3$, 
by using a standard meridian-longitude system, 
we have a one-to-one correspondence between 
the set of slopes on the peripheral torus of $K$ and 
the set of rational numbers, 
or equivalently irreducible fractions, with $1/0$. 
See \cite{R} for example. 
Let $L$ be a 2-bridge link. 
We denote $L(r)$ the manifold obtained by Dehn surgery 
on a component of $L$ along the slope $r \in \mathbb{Q}$, 
i.e., the rational number $r$ corresponds to 
the slope determined by the meridian of the attached solid torus. 

We here recall the classification of exceptional surgery on
a component of a hyperbolic link. 
A Dehn surgery on one component of 
a 2-component hyperbolic link is exceptional, 
i.e., it yields a non-hyperbolic 3-manifold with torus boundary, 
if and only if the obtained manifold 
contains an essential disk, annulus, 2-sphere, or torus. 
See \cite{Th} as the original reference.

Now we give our classification theorem as follows. 

\begin{theorem}
Let $L$ be a hyperbolic 2-bridge link in $S^3$ and 
$L(r)$ denote the 3-manifold 
obtained by Dehn surgery on a component of $L$ along the slope $r$. 
Then the following hold. 
\begin{enumerate}
\item 
$L(r)$ contains neither essential disks nor essential 2-spheres.
\item
$L(r)$ contains an essential torus if and only if 
$L$ is equivalent to $L_{[2w,v,2u]}$ and $r=-w-u$ with
\begin{enumerate}
\item
$w = 1, u = - 1 , |v| \ge 2 $, 
\item
$w \ge 2, |u| \ge 2 , |v| = 1 $. 
\item
$w \ge 2, |u| \ge 2 , |v| \ge 2 $. 
\end{enumerate}
In all the cases, $L(r)$ is never Seifert fibered, and  
$L(r)$ gives a graph manifold if and only if 
the parameters $u,v,w$ satisfies the first and the second conditions. 
\item
$L(r)$ contains an essential annulus, 
but contains no essential tori, 
equivalently $L(r)$ is a small Seifert fibered space 
if and only if 
$L$ is equivalent to 
\begin{enumerate}
\item
$L_{[3,2u+1]}$ and $r=u $,
\item
$L_{[2w+1,3]}$ and $r=-w-1 $,
\item
$L_{[3,-3]}$ and $r=-1$, or, 
\item
$L_{[2w+1,2u+1]}$ and $r=-w+u $
\end{enumerate}
with $w \ge 1, u \ne 0 , -1 $.
\end{enumerate}
\end{theorem}

This theorem will be proved in the last section. 
As a preliminary, we will give a key lemma in the next section.

\medskip

We here recall the known results 
on exceptional surgeries on hyperbolic 2-bridge links. 
These are the motivation of our study, and actually 
our proof of the theorem heavily due to the following.

On hyperbolic 2-bridge knots, 
Brittenham and Wu gave in \cite{BW} 
a complete classification of exceptional surgeries. 
For example, they showed that 
only 2-bridge knots {$K_{[b_1,b_2]}$} admits exceptional surgeries.
Here, by $[a_1, a_2, \cdots,a_n]$, 
we mean a continued fraction expansion 
following \cite{GodaHayashiSong}. 

For 2-bridge links, it follows from the result obtained 
by Wu in \cite{YQWu1999}: 
If a 3-manifold obtained by a Dehn surgery
on a component of a 2-bridge link $L$ 
contains an essential disk, annulus, or 2-sphere, 
then $L$ is equivalent to {$L_{[b_1,b_2]}$}. 
Recall that an embedded disk, annulus, 2-sphere 
in a 3-manifold is called \textit{essential} 
if it is incompressible and not boundary-parallel. 
We remark that 
Dehn surgery on a hyperbolic link 
yielding 3-manifolds with essential 
disk, annulus, or 2-sphere, is 
a typical example of exceptional surgery. 

Further, in \cite{GodaHayashiSong}, 
Goda, Hayashi and Song obtained 
a complete classification (resp. a necessary condition) 
of 2-bridge links on a component of which 
a Dehn surgery yields 
a non-trivial, non-core torus knot exterior
or a cable knot exterior 
(resp. a prime satellite knot exterior) in a lens space.


\section{Surfaces in 2-bridge link exterior}

To prove our theorem, a key investigation is to study 
essential surfaces embedded in 2-bridge link exteriors 
of genus at most one. 
Most parts of such studies have been achieved 
in \cite{GodaHayashiSong}, 
which is based on the machinery of \cite{FloydHatcher}. 
In this section, we give a lemma 
which concerns the remaining cases of \cite{GodaHayashiSong}.

\begin{lemma}\label{lemma}
If a hyperbolic 2-bridge link exterior 
contains a meridionally incompressible essential planer surface $F$
with at most two meridional boundaries on a component of the link 
and non-empty boundary on the other component 
if and only if 
the link is equivalent to $L_{[2,n,-2]}$ with $|n| \ge 2$ 
and $F$ is an essential two punctured disk 
with two meridional punctures on a component on the link 
and a single longitudinal boundary on the other component. 
\end{lemma}

Here a surface $F$ in $E(L)$ 
is called \textit{meridionally incompressible} 
if, for any disk $D \subset S^3$ with $D \cap F = \partial D$ 
and $D$ meeting $L$ transversely in one point in the interior of $D$, 
there is a disk $D' \subset F \cup L$ with $\partial D' = \partial D$, 
$D'$ also meeting $L$ transversely in one interior point.

Actually, in \cite{FloydHatcher}, Floyd and Hatcher studied 
meridionally incompressible essential surfaces in 2-bridge link exteriors, 
and gave a complete description of such surfaces. 
See \cite{FloydHatcher} and \cite{GodaHayashiSong} for details. 
In the following, we assume that the readers are familiar 
to a certain extent.

\begin{proof}[Proof of Lemma]

Let $L = K_1 \cup K_2$ be 
a hyperbolic 2-bridge link in $S^3$, and $E(L)$ its exterior. 

Suppose that there exists 
a meridionally incompressible essential planer surface $F$ in $E(L)$ 
with at most two meridional boundaries on a component of the link, 
say $K_2$, 
and non-empty boundary on the other component $K_1$. 
Then, by \cite[Theorem 3.1 (a)]{FloydHatcher}, 
the surface $F$ is carried by 
a branched surface $\Sigma_\gamma$ 
for some minimal edge-path $\gamma$ 
in the diagram $D_t$ in \cite{FloydHatcher}. 
See also \cite{GodaHayashiSong}. 

Since $F$ has 
meridional boundaries only on $\partial K_2$, 
we see that the minimal edge-path $\gamma$ is in $D_\infty$. 
Moreover, observing the sub-branched surfaces corresponding to 
the edges in the diagram 
depicted in \cite[Figure 3.1]{FloydHatcher} and \cite[Figure 4]{GodaHayashiSong}, 
the edge-path $\gamma$ must consist of 
edges labeled by $B$ or $D$ only. 

Note that the edge-path $\gamma$ connects $1/0$ to $p/q$, 
where $q$ must be even since $L$ is a 2-bridge link, and 
an edge labeled by $D$ can connect the two vertices 
with even denominators. 
Thus, if $\gamma$ contains edges labeled by $B$,  
the edges labeled by $B$ appears in pairs. 
However, 
by observing the shape of the sub-branched surface 
corresponding to the edge labeled by $B$, 
if $\gamma$ contains edges labeled by $B$ in pairs, 
then $F$ would have positive genus, contradicting 
the assumption that $F$ is planer. 
It concludes that the edge-path $\gamma$ consists of 
only edges labeled by $D$. 

Moreover, 
by observing the shape of the sub-branched surface 
corresponding to the edge labeled by $D$, 
the number of meridional boundary components are 
at least the number of the edges labeled by $D$ in $\gamma$. 
Since we are assuming that $F$ has 
at most two meridional boundaries on $\partial N(K_2)$, 
it follows that the length of $\gamma$ is at most two. 

If $\gamma$ is of length one, then $L$ must be equivalent to 
$L_{1/2}$ which is non-hyperbolic, 
contradicting the assumption that $L$ is hyperbolic. 

If $\gamma$ is of length two, then 
the slopes $p/q$ and 1/2 has distance two, 
and so $L$ must be equivalent to 
$L_{[2,n,-2]}$ with some non-zero integer $|n| \ge 2$. 

Conversely, 
if $L$ is equivalent to $L_{[2,n,-2]}$ with $|n| \ge 2$, 
then we can find a two-punctured disk 
naturally spanned by $K_1$. 
It is incompressible or boundary-incompressible, 
otherwise, after compression or boundary-compression, 
we can find  
a meridionally incompressible essential planer surface in $E(L)$ 
with single boundary on $\partial K_2$, 
contradicting that $L$ is hyperbolic in the same way as above. 

This completes the proof. 

\end{proof}


\section{Proof}

In this section, we give a proof of our theorem.

\begin{proof}[Proof of Theorem]

Let $L=K_1 \cup K_2$ be a hyperbolic 2-bridge link in $S^3$ and 
$L(r)$ denote the 3-manifold 
obtained by Dehn surgery on $K_1 \subset L$ along the slope $r$. 
Note that, since the component $K_2$ remains unfilled, 
$L_(r)$ has a torus boundary. 
Also note that it is known by \cite{Menasco} 
that $L$ is hyperbolic if and only if 
$L$ is not equivalent to $L_{1/n}$ for some integer $n$. 

Now suppose that $L(r)$ is non-hyperbolic. 
Then, as remarked before, 
$L (r)$ contains an essential disk, sphere, annulus or torus.

In the following, we give our proof of the theorem 
divided into four claims. 

\begin{claim}\label{claim_sphere}
There are no essential sphere in $L (r)$. 
\end{claim}

\begin{proof}
Suppose for a contrary that there exists an essential sphere in $L(r)$. 
Then, by the standard argument, 
the link exterior $E(L )$ contains 
a connected, orientable, essential 
(i.e., incompressible and $\partial$-incompressible), 
properly embedded planer surface $F$. 
The surface $F$ has 
non-empty boundary components on $\partial N(K_1)$ 
with boundary slope $r$ 
and no boundary components on $\partial N(K_2)$. 

First suppose that $F$ is meridionally incompressible. 
Then, again by \cite[Theorem 3.1 (a)]{FloydHatcher}, 
the surface $F$ is carried by 
a branched surface $\Sigma_\gamma$ 
for some minimal edge-path $\gamma$ 
in the diagram $D_t$ in \cite{FloydHatcher}. 
See also \cite{GodaHayashiSong}. 
In this case, we can apply the argument given 
in \cite[Lemma 12.1]{GodaHayashiSong}. 
Then we see that the minimal edge-path $\gamma$ is in $D_\infty$ 
and is composed of only two edges with label $B$ 
with endpoints $1/0$ and $p/q$, 
where $L_{p/q}$ is equivalent to $L$. 
However, as seen in \cite[Figure 1.1]{FloydHatcher} or 
\cite[Figure 2]{GodaHayashiSong}, 
it implies that $L_{p/q}$ is equivalent to $L_{\pm 1/m}$ for some $m$, 
contradicting $L$ is hyperbolic. 

Next suppose that $F$ is meridionally compressible. 
Perform meridional compressions as possible. 
It can be checked by the standard argument 
that meridional compressions preserve essentiality of surfaces. 
Then, since any boundary curve of 
a meridionally compressing disk is separating on $F$, 
there must exist some component which is 
meridionally incompressible essential planer surface 
with single meridional boundary on $\partial N(K_2)$ 
and with non-empty boundaries on $\partial N(K_1)$. 
However, by Lemma in Section 2, 
such a surface must have 
exactly two meridional boundaries on $\partial N(K_2)$. 
A contradiction occurs. 
\end{proof}

\begin{claim}\label{claim_disk}
There are no essential disk in $L (r)$. 
\end{claim}

\begin{proof}
Suppose for a contrary that there exists an essential disk in $L(r)$. 
It follows that there is a compressible disk for $\partial L(r)$ in $L(r)$. 
By compression, $L(r)$ must be a solid torus. 
Otherwise we would have an essential sphere in $L(r)$ 
contradicting Claim \ref{claim_sphere}.

Then, considering the exterior of $K_2$, 
we can regard $K_1$ as a knot in a handlebody. 
Since the surgery on $K_1$ yields a solid torus again, 
by the result given in \cite{Gabai}, 
$K_1$ is either a 0 or 1-bridge braid in the solid torus $E(K_2)$. 
Then, together with the result of \cite[Proposition 3.2]{MiyazakiMotegi}, 
$K_1$ must be knotted in $S^3$. 
This contradicts that $L$ is a 2-bridge link. 
\end{proof}

\begin{claim}\label{claim_torus}
There exists an essential torus in $L(r)$ 
if and only if 
$L$ is equivalent to $L_{[2w,v,2u]}$ and $r=-w-u$ with
\begin{enumerate}
\item
$w = 1, u = - 1 , |v| \ge 2 $, 
\item
$w \ge 2, |u| \ge 2 , |v| = 1 $. 
\item
$w \ge 2, |u| \ge 2 , |v| \ge 2 $. 
\end{enumerate}
In all the cases, $L(r)$ is never Seifert fibered, and  
$L(r)$ gives a graph manifold if and only if 
the parameters $u,v,w$ satisfies the first and the second conditions. 
\end{claim}

\begin{proof}
Suppose that 
there exists an essential torus in $L (r)$. 

As seen in the proof of Claim \ref{claim_sphere}, 
the link exterior $E(L )$ contains 
a connected, orientable, essential 
properly embedded surface $F$ of genus one 
with non-empty boundaries on $\partial N(K_1)$ 
with boundary slope $r$ 
and no boundary components on $\partial N(K_2)$. 

First suppose that $F$ is meridionally incompressible. 
Then, by \cite[Theorem 3.1 (a)]{FloydHatcher}, 
the surface $F$ is carried by 
a branched surface $\Sigma_\gamma$ 
for some minimal edge-path $\gamma$ 
in the diagram $D_t$ in \cite{FloydHatcher}. 
See also \cite{GodaHayashiSong}. 
Again we can apply the argument given 
in \cite[Lemma 12.1]{GodaHayashiSong}. 
Then, in this case, $\gamma$ has length 4 in $D_\infty$ 
with endpoints $1/0$ and $p/q$, 
where $L_{p/q}$ is equivalent to $L$. 
As claimed in 
the proof of \cite[Theorem 1.5]{GodaHayashiSong}, 
$L_{p/q}$ must be equivalent to 
$L_{[2w,v,2u]}$ with $w \ge 2, |v| \ge 1, |u| \ge 2$. 

It remains to show that 
$L_{[2w,v,2u]}$ actually contains essential torus 
for $w \ge 2, |v| \ge 1, |u| \ge 2$. 
%
%
We here imitate the arguments used in 
the proofs of \cite[Theorem 5.1]{YQWu1999} and 
\cite[Theorem 11.1]{GodaHayashiSong}. 
By performing a band sum of $K_2$ and the curve parallel to 
the one on $\partial N(K_1)$ with slope $r = -w-u$, 
equivalently, using a Kirby move on the framed knot $(K_1,r)$, 
it can be checked directly from the illustration 
that the surgered manifold $L_{[2w,v,2u]} (r)$ 
is homeomorphic to the exterior of a satellite knot 
with a torus knot as a companion in a lens space. 
See Figure \ref{Fig1}. 

\begin{figure}[htb]
\unitlength 0.1in
\begin{picture}( 52.0000, 24.0000)(  2.0000,-34.0000)
%
\special{pn 8}%
\special{ar 800 2200 600 1200  1.5707963 4.7123890}%
%
\special{pn 8}%
\special{ar 4800 2200 600 1200  4.7123890 6.2831853}%
\special{ar 4800 2200 600 1200  0.0000000 1.5707963}%
%
\special{pn 8}%
\special{pa 4800 1000}%
\special{pa 800 1000}%
\special{fp}%
\special{pa 800 3400}%
\special{pa 4800 3400}%
\special{fp}%
%
\special{pn 8}%
\special{ar 4000 2200 610 160  0.0000000 6.2831853}%
%
\special{pn 8}%
\special{ar 1800 2200 610 160  0.0000000 6.2831853}%
%
\special{pn 20}%
\special{pa 800 3200}%
\special{pa 1200 3200}%
\special{fp}%
\special{pa 1200 3200}%
\special{pa 1400 2320}%
\special{fp}%
\special{pa 1600 3400}%
\special{pa 1800 2360}%
\special{fp}%
%
\special{pn 20}%
\special{pa 2000 3400}%
\special{pa 2200 2320}%
\special{fp}%
\special{pa 2400 3400}%
\special{pa 2600 3200}%
\special{fp}%
\special{pa 2600 3200}%
\special{pa 3400 3200}%
\special{fp}%
\special{pa 3400 3200}%
\special{pa 3600 2320}%
\special{fp}%
\special{pa 3800 3400}%
\special{pa 4000 2370}%
\special{fp}%
\special{pa 4200 3400}%
\special{pa 4400 2330}%
\special{fp}%
\special{pa 4600 3400}%
\special{pa 4800 3200}%
\special{fp}%
\special{pa 4800 3200}%
\special{pa 5200 2200}%
\special{fp}%
\special{pa 5200 2200}%
\special{pa 4800 1200}%
\special{fp}%
\special{pa 4800 1200}%
\special{pa 800 1200}%
\special{fp}%
\special{pa 800 1200}%
\special{pa 400 2200}%
\special{fp}%
\special{pa 400 2200}%
\special{pa 800 3200}%
\special{fp}%
%
\special{pn 8}%
\special{pa 880 1320}%
\special{pa 540 2190}%
\special{fp}%
\special{pa 540 2190}%
\special{pa 900 3090}%
\special{fp}%
\special{pa 900 3090}%
\special{pa 1090 3090}%
\special{fp}%
\special{pa 1090 3090}%
\special{pa 1290 2290}%
\special{fp}%
\special{pa 1490 3400}%
\special{pa 1690 2360}%
\special{fp}%
\special{pa 1890 3400}%
\special{pa 2090 2340}%
\special{fp}%
\special{pa 2290 3400}%
\special{pa 2600 3090}%
\special{fp}%
\special{pa 2600 3090}%
\special{pa 3320 3090}%
\special{fp}%
\special{pa 3320 3090}%
\special{pa 3530 2300}%
\special{fp}%
\special{pa 3700 3400}%
\special{pa 3900 2360}%
\special{fp}%
\special{pa 4090 3400}%
\special{pa 4290 2340}%
\special{fp}%
\special{pa 4460 3400}%
\special{pa 4660 3200}%
\special{fp}%
\special{pa 4660 3200}%
\special{pa 5060 2200}%
\special{fp}%
\special{pa 5060 2200}%
\special{pa 4710 1320}%
\special{fp}%
\special{pa 4710 1320}%
\special{pa 880 1320}%
\special{fp}%
%
\special{pn 8}%
\special{pa 1290 2290}%
\special{pa 1490 3400}%
\special{dt 0.045}%
\special{pa 1690 2360}%
\special{pa 1890 3400}%
\special{dt 0.045}%
\special{pa 2090 2340}%
\special{pa 2290 3400}%
\special{dt 0.045}%
\special{pa 3530 2310}%
\special{pa 3700 3400}%
\special{dt 0.045}%
\special{pa 3900 2360}%
\special{pa 4100 3400}%
\special{dt 0.045}%
\special{pa 4300 2340}%
\special{pa 4470 3400}%
\special{dt 0.045}%
%
\special{pn 13}%
\special{pa 1400 2330}%
\special{pa 1600 3400}%
\special{dt 0.045}%
\special{pa 1470 2700}%
\special{pa 1470 2700}%
\special{dt 0.045}%
%
\special{pn 13}%
\special{pa 1800 2360}%
\special{pa 2000 3400}%
\special{dt 0.045}%
\special{pa 2200 2330}%
\special{pa 2400 3400}%
\special{dt 0.045}%
\special{pa 3600 2330}%
\special{pa 3800 3400}%
\special{dt 0.045}%
\special{pa 4000 2360}%
\special{pa 4200 3400}%
\special{dt 0.045}%
\special{pa 4400 2330}%
\special{pa 4600 3400}%
\special{dt 0.045}%
%
\special{pn 20}%
\special{pa 2600 1600}%
\special{pa 2800 2200}%
\special{fp}%
\special{pa 2800 1600}%
\special{pa 3000 2200}%
\special{fp}%
\special{pa 3000 1600}%
\special{pa 3200 2200}%
\special{fp}%
\special{pa 3200 1600}%
\special{pa 3130 1810}%
\special{fp}%
\special{pa 3000 2200}%
\special{pa 3070 1950}%
\special{fp}%
\special{pa 2870 1950}%
\special{pa 2800 2200}%
\special{fp}%
\special{pa 3000 1600}%
\special{pa 2930 1810}%
\special{fp}%
\special{pa 2800 1600}%
\special{pa 2730 1810}%
\special{fp}%
\special{pa 2600 2200}%
\special{pa 2670 1950}%
\special{fp}%
%
\special{pn 20}%
\special{pa 2600 2200}%
\special{pa 2600 2600}%
\special{fp}%
\special{pa 2600 2600}%
\special{pa 2400 2800}%
\special{fp}%
\special{pa 2400 2800}%
\special{pa 2160 2800}%
\special{fp}%
\special{pa 2080 2800}%
\special{pa 2030 2800}%
\special{fp}%
\special{pa 1970 2800}%
\special{pa 1760 2800}%
\special{fp}%
\special{pa 1680 2800}%
\special{pa 1630 2800}%
\special{fp}%
\special{pa 1570 2800}%
\special{pa 1340 2800}%
\special{fp}%
\special{pa 1260 2800}%
\special{pa 1200 2800}%
\special{fp}%
\special{pa 1110 2800}%
\special{pa 910 2800}%
\special{fp}%
\special{pa 910 2800}%
\special{pa 910 1800}%
\special{fp}%
\special{pa 910 1800}%
\special{pa 1110 1600}%
\special{fp}%
\special{pa 1110 1600}%
\special{pa 2600 1600}%
\special{fp}%
%
\special{pn 20}%
\special{pa 3200 1600}%
\special{pa 4500 1600}%
\special{fp}%
\special{pa 4500 1600}%
\special{pa 4700 1800}%
\special{fp}%
\special{pa 4700 1800}%
\special{pa 4700 2800}%
\special{fp}%
\special{pa 4700 2800}%
\special{pa 4380 2800}%
\special{fp}%
\special{pa 4270 2800}%
\special{pa 4240 2800}%
\special{fp}%
\special{pa 4150 2800}%
\special{pa 3980 2800}%
\special{fp}%
\special{pa 3880 2800}%
\special{pa 3850 2800}%
\special{fp}%
\special{pa 3770 2800}%
\special{pa 3550 2800}%
\special{fp}%
\special{pa 3460 2800}%
\special{pa 3430 2800}%
\special{fp}%
\special{pa 3350 2800}%
\special{pa 3200 2600}%
\special{fp}%
\special{pa 3200 2600}%
\special{pa 3200 2200}%
\special{fp}%
%
\special{pn 8}%
\special{sh 0.300}%
\special{pa 2530 2680}%
\special{pa 2530 2680}%
\special{pa 2530 3160}%
\special{pa 2450 3240}%
\special{pa 2450 2750}%
\special{pa 2530 2680}%
\special{fp}%
%
\special{pn 4}%
\special{pa 2456 3236}%
\special{pa 2450 3230}%
\special{fp}%
\special{pa 2486 3206}%
\special{pa 2450 3170}%
\special{fp}%
\special{pa 2516 3176}%
\special{pa 2450 3110}%
\special{fp}%
\special{pa 2530 3130}%
\special{pa 2450 3050}%
\special{fp}%
\special{pa 2530 3070}%
\special{pa 2450 2990}%
\special{fp}%
\special{pa 2530 3010}%
\special{pa 2450 2930}%
\special{fp}%
\special{pa 2530 2950}%
\special{pa 2450 2870}%
\special{fp}%
\special{pa 2530 2890}%
\special{pa 2450 2810}%
\special{fp}%
\special{pa 2530 2830}%
\special{pa 2450 2750}%
\special{fp}%
\special{pa 2530 2770}%
\special{pa 2482 2722}%
\special{fp}%
\special{pa 2530 2710}%
\special{pa 2514 2694}%
\special{fp}%
\end{picture}%
\caption{band sum for $L_{[6,3,6]}$ }\label{Fig1}
\end{figure}
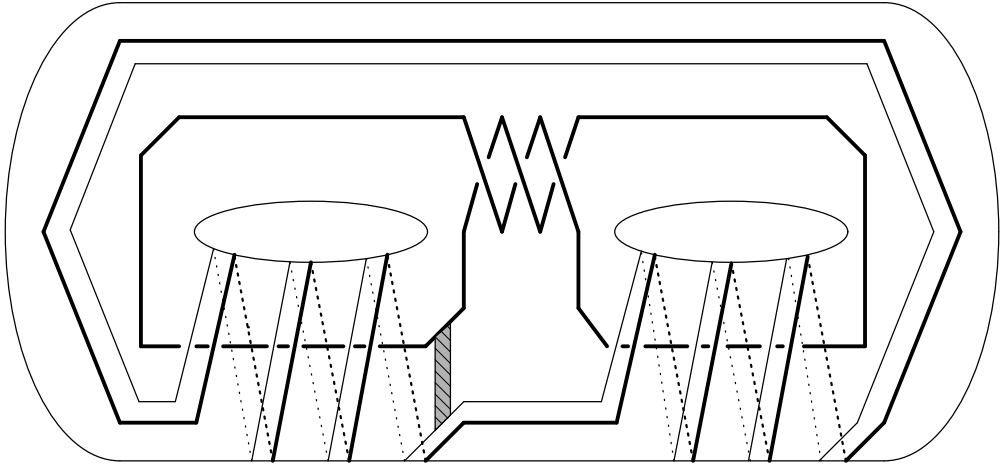

Moreover, 
in the case where $|v| \ne 1$ (resp. $|v| = 1$), 
we can see that 
the companion knot is a torus knot 
and the pattern knot is a hyperbolic knot (resp. a cable knot). 
See also \cite[Theorem 11.1]{GodaHayashiSong} 
in the case where $|v| = 1$. 
Note that we have 
$L_{[2w,\pm 1,2u]} (-w-u) \equiv L_{[2w'+1,2u'+1]} (-w'+u' \pm 1)$ 
for some $w'$ and $u'$.

Next suppose that $F$ is meridionally compressible. 
As in the proof of Claim \ref{claim_sphere}, 
perform meridional compressions as possible. 
It can be checked by the standard argument 
that meridional compressions preserve essentiality of surfaces. 
If some boundary curve of 
a meridionally compressing disk on $F$ is separating, 
then the same contradiction could occur as in Claim \ref{claim_sphere}, 
and so, there must be single meridional compression for $F$ 
along the non-separating curve on $F$. 
Then, by Lemma in Section 2, 
the link is equivalent to $L_{[2,n,-2]}$ with $|n| \ge 2$ 
and $F$ is an essential two punctured disk 
with two meridional punctures on $\partial N(K_2)$ 
and a single longitudinal boundary on $\partial N(K_1)$. 
Actually, by tubing operation, 
we can find a once-punctured torus or klein bottle 
embedded in $E(L)$ 
coming from a spanning surface for $K_1$.

Conversely, 
we can see that 0-surgery on $K_1 \subset L_{[2,n,-2]}$ with $|n| \ge 2$ 
gives the exterior of a knot $K'_2$ in $S^2 \times S^1$. 
This $K'_2$ intersects the level horizontal sphere 
in $S^2 \times S^1$ transversely twice. 
This implies that $E(K_2')$ contains a meridional annulus $A$. 
The annulus $A$ is incompressible 
otherwise the meridian of $K_2'$ bounds a disk in $E(K_2')$, 
contradicting Claim \ref{claim_disk}. 
Also $A$ is not boundary parallel since it is non-separating. 
Thus we conclude that the surgered manifold $L_{[2,n,-2]} (0)$ 
contains an essential annulus $A$. 

From the annulus, by tubing operation, we have 
a non-separating torus or klein bottle, 
which is incompressible by Claim \ref{claim_sphere} 
in the knot exterior. 

It remains that 
$L_{[2,n,-2]} (0)$ is not a Seifert fibered space but a graph manifold. 
Then, since $A$ is essential, 
$A$ must be isotoped so that $A$ is a union of Seifert fibers. 
Along this annulus $A$, we cut $L_{[2,n,-2]} (0)$ open 
to get a compact manifold, say $X_n$, which is the exterior of 
a pair of properly embedded arcs $t \cup t'$ in $S^2 \times [0,1]$. 
See Figure \ref{Fig2}. 
Actually we can see that $X_n$ is homeomorphic to 
the $(2,n)$-torus link exterior, and 
the copies of $A$ appear as 
meridional annuli on the boundary $\partial X_n$ of the knot exterior. 

\begin{figure}[htb]
\unitlength 0.1in
\begin{picture}( 20.4000, 20.4000)(  1.8000,-32.2000)
%
\special{pn 8}%
\special{ar 1200 2200 448 448  0.0000000 6.2831853}%
%
\special{pn 20}%
\special{pa 1000 1800}%
\special{pa 1000 1600}%
\special{fp}%
\special{pa 1400 1800}%
\special{pa 1400 1600}%
\special{fp}%
\special{pa 1000 1600}%
\special{pa 1000 1600}%
\special{fp}%
%
\special{pn 20}%
\special{pa 1000 1400}%
\special{pa 1000 1200}%
\special{fp}%
\special{pa 1400 1400}%
\special{pa 1400 1200}%
\special{fp}%
\special{pa 1000 1200}%
\special{pa 1000 1200}%
\special{fp}%
%
\special{pn 20}%
\special{pa 1400 1600}%
\special{pa 1300 1400}%
\special{fp}%
%
\special{pn 20}%
\special{pa 1300 1600}%
\special{pa 1200 1400}%
\special{fp}%
%
\special{pn 20}%
\special{pa 1100 1600}%
\special{pa 1000 1400}%
\special{fp}%
%
\special{pn 20}%
\special{pa 1200 1600}%
\special{pa 1100 1400}%
\special{fp}%
%
\special{pn 20}%
\special{pa 1000 1600}%
\special{pa 1030 1530}%
\special{fp}%
%
\special{pn 20}%
\special{pa 1200 1600}%
\special{pa 1230 1530}%
\special{fp}%
%
\special{pn 20}%
\special{pa 1100 1600}%
\special{pa 1130 1530}%
\special{fp}%
%
\special{pn 20}%
\special{pa 1300 1600}%
\special{pa 1330 1530}%
\special{fp}%
%
\special{pn 20}%
\special{pa 1370 1470}%
\special{pa 1400 1400}%
\special{fp}%
%
\special{pn 20}%
\special{pa 1170 1470}%
\special{pa 1200 1400}%
\special{fp}%
%
\special{pn 20}%
\special{pa 1060 1470}%
\special{pa 1090 1400}%
\special{fp}%
%
\special{pn 20}%
\special{pa 1260 1470}%
\special{pa 1290 1400}%
\special{fp}%
%
\special{pn 8}%
\special{ar 1200 2200 1020 1020  0.0000000 6.2831853}%
\end{picture}%
\caption{embedded arcs $t \cup t'$ in $S^2 \times [0,1]$. ($n=4$)}\label{Fig2}
\end{figure}
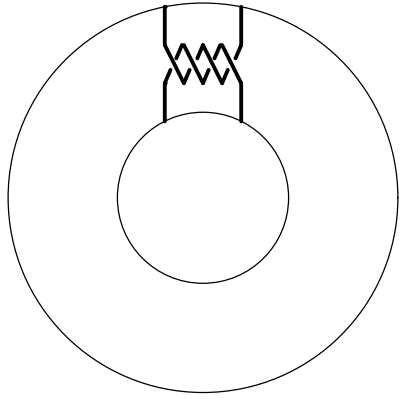

Thus we could actually verify that $X_n$ is Seifert fibered, 
but, in the case where $|n| \ge 2$, 
such annuli cannot be a union of Seifert fibers in $X_n$. 
This means that 
the surgered manifold $L_{[2,n,-2]} (0)$ is 
not a Seifert fibered space but a graph manifold. 
\end{proof}

\begin{claim}\label{claim_annulus}
There exists an essential annulus,
but no essential torus in $L(r)$ 
if and only if $L(r)$ is a small Seifert fibered space 
and $L$ is equivalent to 
\begin{enumerate}
\item
$L_{[3,2u+1]}$ and $r=u $,
\item
$L_{[2w+1,3]}$ and $r=-w-1 $,
\item
$L_{[3,-3]}$ and $r=-1$, or, 
\item
$L_{[2w+1,2u+1]}$ and $r=-w+u $
\end{enumerate}
with $w \ge 1, u \ne 0 , -1 $.
\end{claim}

\begin{proof}
Suppose that 
there exists an essential annulus but no essential torus in $L (r)$. 
Then it is known that $L (r)$ must be a small Seifert fibered space. 

Let $r_2$ be the slope on $\partial N(K_2)$ 
determined by the boundary of the essential annulus. 
Then it is shown that $r_2 \ne 1/0$ as follows. 
Suppose for a contrary that $r_2 = 1/0$, i.e., $r_2$ is meridional. 
Now we are assuming that $L(r)$ is a Seifert fibered space, and 
the essential annulus coming from the surface $F$ must be vertical. 
This implies that the meridian of $K_2$ is a regular fiber of 
the Seifert fibration of $L(r)$. 
Then, as shown in \cite[Proof of Corollary 2.6]{Saito}, 
$K_2$ must be a core knot in the lens space. 
However it contradicts that $L(r)$ is not a solid torus as claimed before. 

Thus we see that $r_2 \ne 1/0$. 
Then, as also shown in \cite[Proof of Corollary 2.6]{Saito}, 
$K_2$ gives a non-trivial non-core torus knot in a lens space. 
In this case, if we perform suitable surgery on $K_2$, 
we have a reducible manifold, 
equivalently, 
a suitable surgery on the 2-bridge link $L$ yields a reducible manifold. 
Then, as a consequence of \cite[Theorem 5.1]{YQWu1999}, 
$L$ must be equivalent to a 2-bridge link corresponding to 
a continued irreducible fraction of length two. 

Now we can apply \cite[Theorem 11.1]{GodaHayashiSong}, 
which establishes a complete classification of 
such 2-bridge links and surgery slopes on which surgeries yield 
non-trivial non-core torus knots in lens spaces. 
This gives us the desired conclusions. 

\end{proof}

By these claims, we have obtained 
our classification of exceptional Dehn surgeries on 
components of hyperbolic two-bridge links.
\end{proof}

\section*{Acknowledgements}
The author would like to thank 
Chuichiro Hayashi and Kimihiko Motegi for helpful discussions.


\bibliographystyle{amsplain}

\end{document}